\documentclass[10pt, a4paper]{amsart}  
\usepackage{amsmath,amssymb,amsfonts}
\usepackage[hidelinks]{hyperref}
\usepackage[initials,nobysame]{amsrefs}
\newcommand{\R}{\mathbb{R}}
\newcommand{\N}{\mathbb{N}}
\newcommand{\Z}{\mathbb{Z}}

\def\spheresdp/{A}
\def\spacesdp/{B}

\newtheorem{defin}{Definition}[section]

\newtheorem{proposition}[defin]{Proposition}
\newtheorem{theorem}[defin]{Theorem}

\newtheorem{lemma}[defin]{Lemma}

\newtheorem{innercustomthm}{Theorem}
\newenvironment{customthm}[1]
  {\renewcommand\theinnercustomthm{#1}\innercustomthm}
  {\endinnercustomthm}

\begin{document}

\title{Optimal densities of packings consisting of highly unequal objects}
\author{\href{http://www.daviddelaat.nl}{David de Laat}}

\address{D.~de Laat, CWI, Science Park 123, 1098 XG Amsterdam, The Netherlands} 
\email{mail@daviddelaat.nl}

\thanks{The author was supported by Vidi grant
  639.032.917 from the Netherlands Organization for Scientific
  Research (NWO)}



\date{\today}

\begin{abstract} 
Let $\Delta$ be the optimal packing density of $\R^n$ by unit balls. We show the optimal packing density using two sizes of balls approaches $\Delta + (1 - \Delta) \Delta$ as the ratio of the radii tends to infinity. More generally, if $B$ is a body and $D$ is a finite set of bodies, then the optimal density $\smash{\Delta_{\{rB\} \cup D}}$ of packings consisting of congruent copies of the bodies from $\{rB\} \cup D$ converges to $\Delta_D + (1 - \Delta_D) \Delta_{\{B\}}$ as $r$ tends to zero.
\end{abstract}

\maketitle

\section{Introduction}

There has been extensive research into the determination of optimal monodisperse packing densities. A well-known example is Hales's proof of the Kepler conjecture on the optimal sphere packing density in $\R^3$ \cites{Hales2005a}. More recently, packings with polydispersity have been investigated: New lower and upper bounds for the density of packings of spheres using several sizes have been given in respectively \cite{HopkinsStillingerTorquato2012} and \cite{LaatOliveiraVallentin2012}. In applications, sphere packings can be used to model many-particle systems, and here it is important to also consider polydispersity as this can ``dramatically affect the microstructure and the effective properties of the materials'' \cite{UcheStillingerTorquato2004}. In this note we discuss the case of wide dispersity; that is, the case where the size ratio of the larger to the smaller objects grows large. One would expect boundary behavior to become negligible as the ratio of the radii tends to infinity, which intuitively means the density converges to $\Delta + (1 - \Delta) \Delta$; see for example \cite{Torquato2001}. To the best of our knowledge a proof of this has not yet been published. Here we provide such a proof which uses standard techniques albeit it is not trivial. 

We prove the following theorem:
\begin{customthm}{4.2}
Suppose $B$ is a body and $D$ is a finite set of bodies. Then,
\[
\lim_{r \downarrow 0} \Delta_{\{rB\} \cup D}
= \Delta_D + (1 - \Delta_D) \Delta_{\{B\}}.
\]
\end{customthm}
\noindent
Here, given a set of bodies $D$, we denote by $\Delta_D$ the optimal packing density using the bodies from $D$. If we take $B$ to be the unit disk in $\R^2$ and $D = \{B\}$, then the theorem says the optimal packing density using two sizes of disks converges to $0.9913\ldots$ as the ratio of the radii goes to infinity. By using that the optimal sphere packing density in $\R^3$ is known to be $\pi/(3\sqrt{2}) = 0.7404\ldots$, the theorem says the optimal packing density using two sizes of balls converges to $0.9326\ldots$ as the ratio of the radii goes to infinity.

\section{Packings and density}

We define a \emph{body} to be a bounded subset of $\R^n$ that has nonempty interior and whose boundary has Lebesgue measure zero. Such a set is Jordan measurable, which means its Lebesgue measure can be approximated arbitrarily well by the volume of inner and outer approximations by finite unions of $n$-dimensional rectangles \cite{Bogachev2007}. Moreover, since the interior of a body is nonempty it contains a ball with strictly positive radius and hence has strictly positive Lebesgue measure.

A \emph{packing} using a set of bodies $D$ is a set of congruent copies of the elements in $D$ such that the interiors of the copies are pairwise disjoint. In other words, a packing is of the form 
\[
P = \Big\{ R_i B_i + t_i : i \in \N, R_i \in O(n), \, t_i \in \R^n, \, B_i \in D\Big\},
\]
where $(R_i B_i^\circ + t_i) \cap (R_j B_j^\circ + t_j) = \emptyset$ for all $i \neq j$. Here $B_i^\circ$ denotes the interior of $B_i$, and $O(n)$ is the orthogonal group. Define $\Sigma_D$ to be the set of packings that use bodies from $D$ and $\Lambda_D$ the set of packings $P \in \Sigma_D$ that have \emph{rational box periodicity}; that is, for which there exists a $p \in \mathbb Q$ such that $|P|+pe_i = |P|$ for all $i \in [n]$. Here $|P| = \bigcup P$ denotes the \emph{carrier} of $P$ and $e_i$ is the $i$th unit vector. 

The \emph{density} and \emph{upper density} (provided these exist) of a set $S \subseteq \R^n$ are defined as
\[
\rho(S) = \lim_{r \to \infty} \frac{\lambda(S \cap r C)}{r^n} \quad \text{and} \quad \overline\rho(S) = \limsup_{r \to \infty} \frac{\lambda(S \cap r C)}{r^n}.
\]
Here $\lambda$ is the Lebesgue measure on $\R^n$, and $C$ is the axis-aligned unit cube centered about the origin. The upper density $\bar\rho(|P|)$ is defined for every $P \in \Sigma_D$, because for each $r > 0$, the set $|P| \cap r C$ is Lebesgue measurable with measure at most $r^n$. The density $\rho(|P|)$ is defined for every $P \in \Lambda_D$: Let $p \in \mathbb Q$ be a period of $P$, then $c = \lambda(|P| \cap kpC)/(kp)^n$ does not depend on $k \in \N$, and for $r$ inbetween $kp$ and $(k+1)p$ we have 
\[
\frac{\lambda(|P| \cap kpC)}{((k+1)p)^n} \leq \frac{\lambda(|P| \cap rC)}{r^n} \leq \frac{\lambda(|P| \cap kpC)}{((k+1)p)^n} + \frac{((k+1)p)^n - (kp)^n}{((k+1)p)^n},
\]
where both the rightmost term and $|\frac{\lambda(|P| \cap kpC)}{((k+1)p)^n} - c|$ converge to $0$ as $k \to \infty$.

We define the \emph{optimal packing density} for packings that use bodies from $D$ by
\[
\Delta_D = \sup_{P \in \Sigma_D} \overline\rho(|P|) = \sup_{P \in \Lambda_D} \rho(|P|).
\]
The second equality follows because for each $P \in \Sigma_D$, we can construct a packing from $\Lambda_D$ whose density is arbitrarily close to $\overline\rho(|P|)$ by taking the subpacking contained in a sufficiently large cube and tiling space with this part of the packing. One might wonder whether the optimal density depends on $C$ being a cube, but it follows from \cite{Groemer1963} that the optimal density $\Delta_D$ is also equal to 
\[
\sup_{P \in \Lambda_D} \lim_{r \to \infty} \lambda(|P| \cap (rB+t))/\lambda(rB),
\]
where $t$ is any point in $\R^n$ and where $B$ is any compact set that is the closure of its interior and contains the origin in its interior.

%

\section{Approximating the interstitial space of a packing}

We first show that a packing, and hence the  interstitial space of a packing, can be approximated uniformly by grid cubes. Given $S \subseteq \R^n$ and $k \in \Z$, define the packings
\[
G_k(S) = \Big\{ C_{k,t} : t \in \Z^n, \, C_{k,t} \subseteq S \Big\} \quad \text{and} \quad G^k(S) = \Big\{ C_{k,t} : t \in \Z^n, \, C_{k,t} \cap S \neq \emptyset \Big\},
\]
where $C_{k,t}$ is the cube $[\frac{t_1}{2^k}, \frac{t_1+1}{2^k}] \times \cdots \times [\frac{t_n}{2^k}, \frac{t_n+1}{2^k}]$ having side length $2^{-k}$.
Given a set $P$ of subsets of $\R^n$, let $P^c = \R^n \setminus |P|$. 

\begin{lemma}
\label{lem:cube approx}
Let $D$ be a finite set of bodies. Then
\[
\rho(|G_k(|P|)|) \uparrow \rho(|P|) \quad \text{and} \quad \rho(|G^k(|P|)|) \downarrow \rho(|P|) 
\]
and hence
\[
\rho(|G_k(P^c)|) \uparrow 1-\rho(|P|) \quad \text{and} \quad \rho(|G^k(P^c)|) \downarrow 1-\rho(|P|) 
\]
as $k \to \infty$ for $P \in \Lambda_D$ uniformly.  
\end{lemma}
\begin{proof}
Let $\varepsilon > 0$ and $P \in \Lambda_D$. Since $P$ has rational box periodicity, the packings $G_k(|P|)$ and $G^k(|P|)$ have rational box periodicity, which means the densities $\rho(|G_k(|P|)|)$ and $\rho(|G^k(|P|)|)$ are defined. For each $k \in \Z$ we have 
\[
|G_k(|P|)| \subseteq |P| \subseteq |G^k(|P|)|,
\]
hence 
\[
\rho(|G_k(|P|)|) \leq \rho(|P|) \leq \rho(|G^k(|P|)|).
\]

We have
\[
\rho(|G_k(|P|)|) = \lim_{r \to \infty} \frac{\lambda(|G_k(|P|)| \cap r C)}{r^n}  \geq \lim_{r \to \infty} \frac{1}{r^n} \sum_{B \in P : B \subseteq rC} \lambda(|G_k(B)|)
\]
and
\[
\rho(|G^k(|P|)|) = \lim_{r \to \infty} \frac{\lambda(|G^k(|P|)| \cap r C)}{r^n} \leq \lim_{r \to \infty} \frac{1}{r^n} \sum_{B \in P : B \cap rC \neq \emptyset} \lambda(|G^k(B)|).
\]
Every $B \in D$ is Jordan measurable, so there exists a number $K = K(B, \varepsilon)$ such that
\[
\lambda(|G_k(B)|) \geq \lambda(B) - \varepsilon \quad \text{and} \quad \lambda(|G^k(B)|) \leq \lambda(B) + \varepsilon
\]
for all $k \geq K$.
Since $D$ is a finite set, this implies
\[
\rho(|G_k(|P|)|) \geq \lim_{r \to \infty} \frac{1}{r^n} \sum_{B \in P : B \subseteq rC} (\lambda(B) - \varepsilon)
\]
and
\[
\rho(|G^k(|P|)|) \leq \lim_{r \to \infty} \frac{1}{r^n} \sum_{B \in P : B \cap rC \neq \emptyset} (\lambda(B) + \varepsilon)
\]
for all $k \geq \max_{B \in D} K(B,\varepsilon)$.
Since each body $B$ in the finite set $D$ is bounded, there exists a number $r_0 = r_0(D) \geq 0$ such that
\begin{align*}
\lim_{r \to \infty} \frac{1}{r^n} \sum_{B \in P : B \subseteq rC} \lambda(B) 
\geq \lim_{r \to \infty} \frac{\lambda(|P| \cap (r-r_0)C)}{r^n} = \lim_{r \to \infty} \frac{\lambda(|P| \cap rC)}{(r+r_0)^n} = \rho(|P|)
\end{align*}
and
\begin{align*}
\lim_{r \to \infty} \frac{1}{r^n} \sum_{B \in P : B \cap rC \neq \emptyset} \lambda(B) 
\leq \lim_{r \to \infty} \frac{\lambda(|P| \cap (r+r_0)C)}{r^n} = \lim_{r \to \infty} \frac{\lambda(|P| \cap rC)}{(r-r_0)^n}
= \rho(|P|).
\end{align*}
Moreover, each body in the finite set $D$ has nonempty interior, so there exists a constant $c = c(D)$ such that the number of congruent copies of elements from $D$ that fit in a cube of radius $r+r_0$ is at most $c r^n$. Hence,
\[
\lim_{r \to \infty} \frac{1}{r^n} \sum_{B \in P : B \subseteq rC} \varepsilon \leq c\varepsilon \quad \text{and} \quad \lim_{r \to \infty} \frac{1}{r^n} \sum_{B \in P : B \cap rC \neq \emptyset} \varepsilon \leq c\varepsilon.
\]
Hence, for all $k \geq \max_{B \in D} K(B,\varepsilon)$ we have
\[
\rho(|G_k(|P|)|) \geq \rho(|P|) - c\varepsilon \quad \text{and} \quad \rho(|G^k(|P|)|) \leq \rho(|P|) + c\varepsilon,
\] 
which implies
\[
\rho(|G_k(|P|)|) \uparrow \rho(|P|) \quad \text{and} \quad \rho(|G^k(|P|)|) \downarrow \rho(|P|),
\]
and hence 
\[
\rho(|G_k(P^c)|) = 1 - \rho(|G^k(|P|)|) \uparrow 1 - \rho(|P|)
\]
and
\[
\rho(|G^k(P^c)|) = 1 - \rho(|G_k(|P|)|) \downarrow 1 - \rho(|P|),
\]
as $k \to \infty$ for $P \in \Lambda_D$ uniformly.
\end{proof} 

\section{Polydisperse packings}

Let $D$ and $D'$ be sets of bodies. Given $P \in \Lambda_D$, define
\[
\Lambda_{D'}(P) = \Big\{ Q \in \Lambda_{D \cup D'} : Q = P \cup R, \, R \in \Lambda_{D'}\Big\}.
\]
The optimal density of such packings is given by  $\Delta_{D'}(P) = \sup_{Q \in \Lambda_{D'}(P)} \rho(|Q|)$. In the following lemmas we give the optimal density given that a part of the packing is already fixed.

\begin{lemma}\label{lem:smalldiameter}
Suppose $D$ is a finite set of bodies. For every $\varepsilon > 0$ there is a scalar $R = R(D, \varepsilon) > 0$ such that
\[
\rho(|P|) + (1-\rho(|P|)) \Delta_{\{B\}} - \varepsilon \leq \Delta_{\{B\}}(P) \leq \rho(|P|) + (1-\rho(|P|)) \Delta_{\{B\}} + \varepsilon
\]
for all $P \in \Lambda_D$ and all bodies $B$ with $\mathrm{diam}(B) \leq R$.
\end{lemma}
\begin{proof}
Let $0 < \varepsilon \leq 1$ and $P \in \Lambda_D$.

By Lemma~\ref{lem:cube approx} there exists an integer $K_1 = K_1(D, \varepsilon)$ such that 
\[
\label{eq:v1}
\rho(|G_k(P^c)|) \geq 1 - \rho(|P|) - \varepsilon/2 \quad \text{for all} \quad k \geq K_1.
\]
By the definition of density there exists a scalar $R_1 = R_1(k, \varepsilon) > 0$ such that for each body $B$ with $\mathrm{diam}(B) \leq R_1$ we can pack each cube in $G_k(P^c)$ with congruent copies of $B$ with density at least $\Delta_{\{B\}} - \varepsilon/2$. By taking the union of the packings of the cubes together with $\{P\}$ we obtain a packing from $\Delta_{\{B\}}(P)$ which has density at least 
\begin{align*}
\rho(|P|) +  \rho(|G_k(P^c)|) (\Delta_{\{B\}} - \varepsilon/2) &\geq \rho(|P|) + (1 - \rho(|P|) - \varepsilon/2)(\Delta_{\{B\}} - \varepsilon/2)\\
&\geq \rho(|P|) + (1-\rho(|P|)) \Delta_{\{B\}} - \varepsilon.
\end{align*}
This implies 
\[
\Delta_{\{B\}}(P) \geq \rho(|P|) + (1-\rho(|P|)) \Delta_{\{B\}} - \varepsilon 
\]
for all $P \in \Lambda_D$ and all bodies $B$ with $\mathrm{diam}(B) \leq R_1$.

By Lemma~\ref{lem:cube approx} there exists an integer $K_2 = K_2(D, \varepsilon)$ such that
\[
\rho(|G^k(P^c)|) \leq 1 - \rho(|P|) + \varepsilon/3 \quad \text{for all} \quad k \geq K_2.
\]
Again, by the definition of density there exists a scalar $R_2 = R_2(k, \varepsilon)$ such that for each body $B$ with $\mathrm{diam}(B) \leq r$ and each $Q \in \Lambda_{\{B\}}(P)$, the intersection of $|Q \setminus P|$ with a cube from $G^k(P^c)$ has density at most $\Delta_{\{B\}} + \varepsilon/3$ in that cube. So, 
\begin{align*}
\rho(|Q|) & = \rho(|P|) + \rho(|Q \setminus P|) \leq \rho(|P|) + \rho(|G^k(P^c)|)(\Delta_{\{B\}} + \varepsilon/3)\\
&\leq \rho(|P|) + (1 - \rho(|P|) + \varepsilon/3)(\Delta_{\{B\}} + \varepsilon/3)\\
&\leq \rho(|P|) + (1 - \rho(|P|))\Delta_{\{B\}} + \varepsilon,
\end{align*}
hence
\[
\Delta_{\{rB\}}(P) \leq \rho(|P|) + (1-\rho(|P|)) \Delta_{\{B\}} + \varepsilon
\]
for all $P \in \Lambda_D$ and all bodies $B$ with $\mathrm{diam}(B) \leq R_2$.

The proof is then complete by setting $R = \min \{R_1, R_2\}$.
\end{proof}

Using the above result the following lemma is immediate.

\begin{lemma}
\label{lem:addballs}
Suppose $B$ is a body and $D$ is a finite set of bodies. Then,
\[
\lim_{r \downarrow 0} \Delta_{\{B_k\}}(P) = \rho(|P|) + (1 - \rho(|P|)) \Delta_{\{B\}} \quad \text{for} \quad P \in \Lambda_D \quad \text{uniformly}.
\]
\end{lemma}

We prove the main result by using the uniform convergence in the above lemma.

\begin{theorem}
\label{thm:polydisperse:main theorem}
Suppose $B$ is a body and $D$ is a finite set of bodies. Then,
\[
\lim_{r \downarrow 0} \Delta_{\{rB\} \cup D}
= \Delta_D + (1 - \Delta_D) \Delta_{\{B\}}.
\]
\end{theorem}
\begin{proof}
We have $\Lambda_{\{rB\} \cup D} = \bigcup_{Q \in \Lambda_D} \Lambda_{\{rB\}}(Q)$, so
\[
\Delta_{\{rB\} \cup D} = \sup_{P \in \Lambda_{\{rB\} \cup D}} \rho(|P|) 
= \sup_{Q \in \Lambda_D} \; \sup_{P \in \Lambda_{\{rB\}}(Q)} \rho(|P|)
= \sup_{Q \in \Lambda_D} \Delta_{\{rB\}}(Q),
\]
and
\[
\lim_{r \downarrow 0} \Delta_{\{rB\} \cup D} = \lim_{r \downarrow 0} \sup_{Q \in \Lambda_D} \Delta_{\{rB\}}(Q).
\]
By Lemma~\ref{lem:addballs}, we have 
\[
\lim_{r \downarrow 0} \Delta_{\{rB\}}(Q) = \rho(|Q|) + (1 - \rho(|Q|)) \Delta_{\{B\}},
\]
and since convergence is uniform for $Q \in \Lambda_D$, we can interchange limit and supremum and obtain
\begin{align*}
\lim_{r \downarrow 0} \sup_{Q \in \Lambda_D} \Delta_{\{rB\}}(Q) 
&= \sup_{Q \in \Lambda_D} \lim_{r \downarrow 0} \Delta_{\{rB\}}(Q) 
= \sup_{Q \in \Lambda_D} (\rho(|Q|) + (1 - \rho(|Q|)) \Delta_{\{B\}})\\
&= \Delta_{\{B\}} + (1 - \Delta_{\{B\}}) \sup_{Q \in \Lambda_D} \rho(|Q|)
= \Delta_{\{B\}} + (1 - \Delta_{\{B\}}) \Delta_D,
\end{align*}
and since $\Delta_{\{B\}} + (1 - \Delta_{\{B\}}) \Delta_D = \Delta_D + (1 - \Delta_D) \Delta_{\{B\}}$ this completes the proof.
\end{proof}

As a special case of the above theorem we obtain the result mentioned in the abstract of this note: We have $\lim_{r \downarrow 0} \Delta_{\{B, r B\}} = \Delta + (1 - \Delta) \Delta$, where $B$ is the closed unit ball and where $\Delta = \Delta_{\{B\}}$. By iteratively applying the above theorem and rewriting the resulting expression we see that if $B_1, \ldots, B_k$ are bodies, then
\begin{align*}
\lim_{r_k \downarrow 0} \cdots \lim_{r_1 \downarrow 0} \Delta_{\{r_1B_1, \ldots, r_kB_k\}}
&= \Delta_{\{B_1\}} + (1-\Delta_{\{B_1\}}) \lim_{r_k \downarrow 0} \cdots \lim_{r_2 \downarrow 0} \Delta_{\{r_2B_2, \ldots, r_kB_k\}} \\
&= 1 -(1 - \Delta_{\{B_1\}}) \cdots (1 - \Delta_{\{B_k\}}).
\end{align*}
Moreover, if $B$ is a body and $D = \{ r B: r > 0\}$, then for every $k$ there is an $r \in \R^k$ such that
\[
\Delta_D \geq 1 - (1-\Delta_{\{B\}})^k - \frac{1}{k},
\]
so we get the intuitive result $\Delta_D = 1$. The following proposition gives a strengthening of this result. Note that the condition $\inf_{B \in D} \Delta_{\{B\}} > 0$ here is always satisfied if we restrict to convex bodies. 
\begin{proposition}
Suppose $D$ is a set containing bodies of arbitrarily small diameter such that $s = \inf_{B \in D} \Delta_{\{B\}} > 0$. Then there exist packings consisting of congruent copies of bodies from $D$ whose density is arbitrarily close to $1$; that is, $\Delta_D = 1$.
\end{proposition}
\begin{proof}
Let $0 < \varepsilon < s$. Select an arbitrary $B_1 \in D$ and choose $P_1 \in \Lambda_{\{B_1\}}$ such that $\rho(|P_1|) \geq \Delta_{\{B_1\}} - \varepsilon$. By Lemma~\ref{lem:smalldiameter} there exists a body $B_2 \in D$ and a packing $P_2 \in \Lambda_{\{B_2\}}(P_1)$ such that
\begin{align*}
\rho(|P_2|) 
&\geq \rho(|P_1|) + (1-\rho(|P_1|)) \Delta_{\{B_2\}} - \varepsilon\\
&= \Delta_{\{B_2\}} + (1 - \Delta_{\{B_2\}}) \rho(|P_1|) -\varepsilon\\
&\geq \Delta_{\{B_1\}} + \Delta_{\{B_2\}} - \Delta_{\{B_1\}} \Delta_{\{B_2\}} - 2\varepsilon\\
&\geq 1 - (1 - \Delta_{\{B_1\}} + \varepsilon)(1 - \Delta_{\{B_2\}} + \varepsilon) \geq 1 - (1 - (s - \varepsilon))^2.
\end{align*}
By continuing like this we see that for each $k$ there exists a packing $P_k \in \Lambda_D$ such that 
\[
\rho(|P_k|) \geq 1 - (1 - (s - \varepsilon))^k,
\]
which completes the proof.
\end{proof}

\end{document}